\documentclass[12pt]{article}
\usepackage[top=1in, bottom=1in, left=1in, right=1in]{geometry}
\usepackage{xcolor}
\usepackage{geometry}
\usepackage{amssymb,bm}
\usepackage{amsmath}
\usepackage{amsthm}
\usepackage{graphicx}
\usepackage{hyperref}
\usepackage{caption}
\usepackage{float} 
\hypersetup{pdfborder=0 0 0}
\usepackage{bbm}
\usepackage{physics}

\newtheorem{theorem}{{\sc Theorem}}[section]

\newtheorem{lemma}[theorem]{{\sc Lemma}}

\newtheorem{remark}[theorem]{Remark}

\newcommand{\RR}{\mathbb{R}}

\newcommand{\n}{\noindent}
\newcommand{\comment}[1]{}

\DeclareMathAlphabet{\pazocal}{OMS}{cmsy}{m}{n}

\newcommand{\xx}{\mathsf{x}}
\newcommand{\yy}{\mathsf{y}}
\newcommand{\zz}{\mathsf{z}}
\newcommand{\ff}{\mathsf{f}}
\newcommand{\pp}{\mathsf{p}}
\newcommand{\hh}{\mathsf{h}}
\newcommand{\ww}{\mathsf{w}}

\title{On the potential lack of response in a model of second-harmonic generation. A computer-assisted proof. }
\author{
Miguel Ayala
\footnote{Department of Mathematics and Statistics, McGill University, Montreal, QC, Canada. {\tt (miguel.ayala@mail.mcgill.ca)}}, 
\ \
Dominic Blanco \footnote{Department of Mathematics, Rutgers University, New Brunswick, NJ, USA (dominic.blanco@rutgers.edu)}, 
\ \ 
Fioralba Cakoni\footnote{Department of Mathematics, Rutgers University, New Brunswick, NJ, USA (fc292@math.rutgers.edu)},  
\\
Narek Hovsepyan\footnote{School of Mathematical and Statistical Sciences, University of Texas Rio Grande Valley, Edinburg, TX, USA (narek.hovsepyan@utrgv.edu)}, \ \ and 
\ \ 
Michael S. Vogelius\footnote{Department of Mathematics, Rutgers University, New Brunswick, NJ, USA (vogelius@math.rutgers.edu)}}
\date{}

\begin{document}
\maketitle

\begin{center}
Dedicated to the memory of Robert V. Kohn and his profound contributions to Continuum Mechanics and Nonlinear Science.
\end{center}

\begin{abstract}
This paper provides a rigorous computer-assisted proof of the existence of generalized transmission eigenvalues arising in nonlinear optics in the context of high-order harmonic generation, a result conjectured in \cite{chlv23}. The analysis is carried out for a one-dimensional nonlinear medium, where the problem reduces to a coupled system of nonlinear homogeneous ordinary differential equations subject to nonstandard boundary conditions. These eigenvalues correspond to probing frequencies $\omega$ for which there exists a nontrivial incident $\omega$-wave such that the second-harmonic field generated does not persist outside the compact support of the nonlinear medium, thereby rendering its nonlinear properties undetectable to an external observer. Building on earlier numerical evidence, we establish that, in the low-frequency regime, there exist generalized transmission eigenvalues whose associated eigenfunctions exhibit blow-up behavior as the frequency tends to zero. The proof combines analytical arguments with validated numerics, employing a Newton–Kantorovich framework together with interval arithmetic to rigorously control approximation errors.  The algorithmic implementation underlying the computer-assisted proof is made available on GitHub \cite{TransmissionEigenvalues.jl}.
\end{abstract}

\section{Introduction}
\setcounter{equation}{0}
Higher-order harmonic generation is a nonlinear model describing higher-order optical harmonic generation in bulk crystals \cite{boyd,mol}. Such media, when probed with monochromatic laser beams, generate waves at new frequencies. This process plays a central role in modern optics, with applications ranging from frequency conversion to imaging and sensing. The most common example is perhaps the green laser pointer, which emits frequency-doubled green light based on an infrared laser source and an internal (second-harmonic generating) crystal. More specifically, the second-harmonic generation (SHG) process arises in crystals with particular symmetries, in which the interaction of an incident wave of frequency $\omega$ with this nonlinear medium produces a wave at frequency $2\omega$ \cite{chlv23}. Mathematically, SHG is described by a system of coupled nonlinear wave equations in the frequency domain with quadratic nonlinearity, where interactions between different frequency components are determined by the material's nonlinear susceptibilities.  In laser technology, SHG is often modeled as a one-dimensional problem by considering incident waves impinging on the nonlinear medium at normal incidence \cite{boyd}.  This work focuses on a one-dimensional model of SHG corresponding to wave propagation in a bounded nonlinear medium, and it is a follow-up to the investigation in \cite{chlv23}. In \cite{chlv23}, the authors investigated the existence of probing frequencies $\omega$  that may yield a vanishing $2\omega$ scattered field, given a second-harmonic generation inhomogeneity of compact support. In other words, frequencies at which the nonlinear effects of the medium can be invisible to an external observer. A necessary condition for the existence of such frequencies is the solvability of a nonlinear eigenvalue problem for a system of two nonlinear PDEs on the support of the nonlinear medium. In the one-dimensional case studied here this necessary condition is also sufficient to guarantee the existence of a nontrivial incident $\omega$-wave which generates no scattered $2\omega$ wave. It should be pointed out that the sufficiency is not true in the higher dimensions, except in special cases, such as a spherically symmetric medium. The two functions associated with the eigenvalue problem represent the total fields at frequency $\omega$ and $2\omega$  linked through non-standard boundary conditions. The corresponding eigenvalues are referred to as generalized second-harmonic transmission eigenvalues, or for short, generalized transmission eigenvalues. In \cite{chlv23} it was conjectured that for a one-dimensional medium a continuum of low-frequency generalized transmission eigenvalues $\omega$ does exist and the corresponding eigenfunctions are unbounded (in an appropriate norm) as $\omega\to 0$.\footnote{We note, however, that taking linear and nonlinear dispersion into account may affect this unbounded behavior of the eigenfunctions (cf. Remark~\ref{REM disp}).} This conjecture was supported by numerical evidence. The goal of the present paper is to provide a rigorous computer-assisted proof of the conjecture. Although not presented here for reasons of brevity, the approach can be extended to the case of a spherically symmetric nonlinear medium in higher dimensions, which leads to the same type of nonlinear eigenvalue problem in the radial variable, with slightly modified boundary conditions. As a result, the existence of a continuum of generalized transmission eigenvalues (not necessarily located in the low-frequency regime) can be established. We note that the generalized transmission eigenvalues for nonlinear media are conceptually related to transmission eigenvalues and non-scattering frequencies in linear scattering theory, where, at such probing frequencies, it is possible to have zero scattering from a given linear inhomogeneity \cite{blas2, blas, CBMS2, nonscattering, nonscattering-A, Hu, nonscattering-SS, HV1, HV2}. It is known that these special frequencies play a role for the unique determination of material properties and for the success of reconstruction algorithms \cite{CBMS2}, as well as for material design. We expect generalized transmission eigenvalues, together with their associated energies, to play a similar role in inverse scattering for nonlinear optics, as well as in nonlinear material design.

\subsection{Preliminaries and the main result}
\setcounter{equation}{0}

Let us consider a nonlinear medium (typically a nonlinear crystal) occupying the interval $D=(0,1)$. Let $E_1$ and $E_2$ denote the $\omega$- and $2\omega$-components of the transmitted electric field inside $D$. In \cite{chlv23}, we asked whether there can exist incident $\omega$-waves for which the nonlinear medium does not scatter the second-harmonic $2\omega$-wave. For such incident waves the nonlinear effects remain localized inside the medium and are invisible to an outside observer. We introduced a nonlinear eigenvalue problem that provides necessary conditions for the existence of such ``second-harmonic nonscattering" frequencies and incident waves:

\begin{equation} \label{TE general}
\begin{cases}
\displaystyle E_1'' + \omega^2 q  E_1  = -\omega^2 \chi_1 E_2 \overline{E}_1 \hspace{0.7in} &\text{in} \ D = (0,1)
\\[.05in]
\displaystyle E_2'' + 4 \omega^2 q E_2   = -\omega^2\chi_2 E_1^2  &\text{in} \ D\\[.05in]
E_2 =  E_2'  = 0 &\text{at} \ \partial D = \{0,1\},
\end{cases}
\end{equation}

\n where $q$ denotes the relative permittivity of the medium, and $\chi_1$ and $\chi_2$ denote the nonlinear susceptibilities of the medium. Specifically, $\chi_1$ describes the interaction between a $2\omega$-wave and an $-\omega$-wave (note the complex conjugation on $E_1$), while $\chi_2$ represents the self-interaction of an $\omega$-wave. In general, these quantities are functions of $x$ (and $\omega$); however, for simplicity, we assume here that they are positive constants. Thus, if there exists a nontrivial incident $\omega$-wave such that the corresponding $2\omega$-wave does not scatter, then \eqref{TE general} admits a nontrivial solution $(E_1,E_2)$. Conversely if \eqref{TE general} admits a nontrivial solution, then the total fields $E_1$ and $E_2$ may be extended to all of $\mathbb{R}$, satisfying $E''_j+\omega^2E_j=0$ outside $D$ and with $E_j$ and $E'_j$ continuous across $\partial D$. Indeed $E_2$ extends to zero outside $D$ and $E_1$ has the extension
$$
E_1=\begin{cases} a^- e^{-i\omega x}+a^+ e^{i \omega x} ~~~\hbox{ for } x<0 \\
b^- e^{-i\omega x}+b^+ e^{i \omega x} ~~~\hbox{ for } x>1 \end{cases}~.
$$
This may be rewritten as
$$
E_1=b^- e^{-i\omega x} +a^+ e^{i \omega x} + 
\begin{cases} (a^- -b^-)e^{-i \omega x} ~~~\hbox{ for } x<0 \\ (b^+-a^+) e^{i \omega x} ~~~~~\hbox{ for } x>1  \end{cases}~.
$$
Here the first part represents the (everywhere defined) incident $\omega$-wave and the second part is the scattered part of the $\omega$-wave, which satisfies the outgoing radiation condition.
In \cite{chlv23}, we showed that a sufficiently small $\omega$ can be an eigenvalue of \eqref{TE general} only if the corresponding eigenfunction blows up at a rate of $1/\omega^{2}$. Namely, if

\begin{equation} \label{e-function blow up}
\|(E_1, E_2)\|_{H^{2}(D)\times L^{2}(D)} \geq \frac{c}{\omega^2},
\end{equation}

\n where $c>0$ depends on the material parameters $q$, $\chi_1$ and $\chi_2$, and $H^2$ denotes the Sobolev space of order 2. Based on numerical evidence, we conjectured that such eigenvalues exist. However, a rigorous treatment of the existence of such (or any other) eigenvalues of \eqref{TE general} remained open. In this work, we use computer-assisted techniques to prove this conjecture and establish the following result:

\begin{theorem} \label{THM main 1}
Assume that $q, \chi_1$ and $\chi_2$ are positive constants. There exist $\delta, c>0$ such that any $\omega \in (0, \delta)$ is an eigenvalue of \eqref{TE general}, with a corresponding eigenfunction $(E_1, E_2)$ satisfying \eqref{e-function blow up}.
\end{theorem}

\n In view of \eqref{e-function blow up}, let us rescale the eigenfunctions and set $u_j=\omega^{2} E_j$ for $j=1,2$. We then obtain the problem

\begin{equation} \label{TE u}
\begin{cases}
\displaystyle u_1'' + \omega^2 q  u_1  = - \chi_1 u_2 \overline{u}_1 \hspace{0.7in} &\text{in} \ (0,1)
\\[.05in]
\displaystyle u_2'' + 4 \omega^2 q u_2   = -\chi_2 u_1^2  &\text{in} \ (0,1)\\[.05in]
u_2 =  u_2'  = 0 &\text{at} \ \{0,1\}.
\end{cases}
\end{equation}

\n The desired result will follow once we show that \eqref{TE u} admits a nontrivial solution for any $\omega\in(0,\delta)$ such that $(u_1,u_2)$ remains bounded away from zero as $\omega\to0$. Let us make further simplifications. Since $\chi_1$ and $\chi_2$ are constants, we may rescale the eigenfunctions and work instead with $\sqrt{\chi_1\chi_2} \, u_1$ and $\chi_1 u_2$. These functions solve \eqref{TE u}, but with $\chi_1=\chi_2=1$. Moreover, instead of $\omega$, let us work with the eigenvalue parameter $\sqrt{q} \omega$. Thus, we may equivalently study the following reduced problem:

\begin{equation} \label{TE u reduced}
\begin{cases}
\displaystyle u_1'' + \omega^2  u_1  = - u_2 \overline{u}_1 \hspace{0.7in} &\text{in} \ (0,1)
\\[.05in]
\displaystyle u_2'' + 4 \omega^2 u_2   = - u_1^2  &\text{in} \ (0,1)\\[.05in]
u_2 =  u_2'  = 0 &\text{at} \ \{0,1\}.
\end{cases}
\end{equation}

\n Regarding this reduced problem, we obtain the following result, which implies Theorem~\ref{THM main 1}:

\begin{theorem} \label{THM main 2}
For any $\omega \in [0, 2.03]$ the problem \eqref{TE u reduced} admits a nontrivial solution $\bm{u} = (u_1, u_2)$ that depends smoothly (in fact, $C^\infty$) on $\omega$. The norm $\|\bm{u}\|_{H^{2}(D)\times L^{2}(D)}$ stays uniformly bounded away from zero, and $\bm{u}$ has the following symmetries: $\Re u_j$ is even and $\Im u_j$ is odd with respect to the midpoint $x = \frac{1}{2}$ for $j=1,2$.
\end{theorem}

\begin{remark} \label{REM disp}
\normalfont

More generally, if $q, \chi_1$ and $\chi_2$ are positive functions of $\omega$ (but are independent of $x$), the same rescaling argument leading to the reduced problem remains valid. Indeed, for any $\omega \sqrt{q(\omega)} \in [0, 2.03]$

\begin{equation*}
\|(E_1, E_2)\|_{H^{2}(D)\times L^{2}(D)} \geq \frac{c}{\omega^2 \max \left\{\chi_1(\omega), \sqrt{\chi_1(\omega) \chi_2(\omega)} \right\}},
\end{equation*}

\n where $c>0$ is some absolute constant independent of all the parameters involved.

\end{remark}

\subsection{Main ideas and  outline of the paper}

\n In \cite{chlv23}, we observed that for sufficiently small $\omega$, the problem \eqref{TE u reduced} (or equivalently \eqref{TE general}) cannot admit a nontrivial real-valued solution. Consequently, both $u_1$ and $u_2$ must be complex-valued functions. This can be seen easily by first setting $\omega = 0$ in \eqref{TE u reduced}, then integrating the second equation and using the boundary conditions to obtain $\int_0^1 u_1^2 dx = 0$, which implies that any nonzero $u_1$ cannot be real-valued. A perturbation argument then yields the result for sufficiently small $\omega$.

The main difficulty in studying the eigenvalue problem \eqref{TE u reduced} is that it possesses a continuous symmetry:

\begin{equation} \label{CS}
\text{if} \ (u_1,u_2) \ \text{solves} \ \eqref{TE u reduced}, \ \text{then so does} \ (e^{i\theta} u_1, e^{i2\theta} u_2) \ \text{for any} \ \theta \in [0,2\pi].
\end{equation}

\n As a result, nontrivial solutions of \eqref{TE u reduced} cannot be isolated, and there is no local uniqueness: in any neighborhood of a nontrivial solution, there exists another one. Consequently, contraction-type arguments cannot be applied. The analysis is made possible by imposing additional symmetries on the eigenfunctions, which break the continuous symmetry and restore local uniqueness. Specifically, we impose

\begin{equation} \label{sym}
u_j\left( \tfrac{1}{2} + x \right) = \overline{u_j \left( \tfrac{1}{2} - x \right)} \qquad \qquad j=1,2 \ \ \text{and} \ \ x \in \left(-\tfrac{1}{2},  \tfrac{1}{2} \right).
\tag{SYM}
\end{equation}

\n In other words, we require the real parts of $u_j$ to be even and the imaginary parts to be odd with respect to the midpoint $\tfrac{1}{2}$. Our motivation for imposing \eqref{sym} comes from preliminary numerical analysis, which showed evidence for the existence of such symmetric solutions. The nontrivial solutions of \eqref{TE u reduced} that lie in the subset defined by the symmetry relation \eqref{sym} do not satisfy \eqref{CS}, and we are thus in a position to use a contraction-type argument combined with computer-assisted methods to establish the existence of nontrivial solutions of \eqref{TE u reduced}.

The idea can be described as follows. The system \eqref{TE u reduced}, together with \eqref{sym}, can be written as an abstract operator equation $F(\bm{u}) = 0$ in an appropriate Banach space, where $\bm{u}=(u_1,u_2)$. We first find an approximate solution $\bm{p}$ of this equation numerically. We then study the operator equation near $\bm{p}$ by writing $\bm{u} = \bm{p} + \bm{h}$, where $\bm{h}$ is small. Linearizing the operator $F$ near $\bm{p}$, we show that this linearization is invertible (as a consequence of the imposed symmetries). As a result, we can rewrite our operator equation as a fixed-point equation $\bm{h} = T(\bm{h})$, where $T$ is a nonlinear operator that we then show is a contraction on a small ball around the origin. This implies the existence of a solution $\bm{h}$, and consequently of a true solution $\bm{u}$ near the approximate solution $\bm{p}$. Moreover, $T$ is a contraction on a ball whose radius is small enough to ensure that the fixed point $\bm{h} \neq -\bm{p}$. In other words, $\bm{u}$ is sufficiently close to $\bm{p}$, guaranteeing that the true solution $\bm{u}$ we obtain is nontrivial. This is the celebrated Newton-Kantorovich approach, which has been widely used in the literature \cite{Arioli2005, polynomial_chaos, Day2007, Plum1992, Yamamoto2021, Ort68}.

We now give an outline of the paper.

\vspace{.05in}

$\bullet$ We rewrite \eqref{TE u reduced} as a first-order real system, in which case the unknown function becomes an 8-dimensional vector, denoted by $\bm{x} = (x_1,...,x_8)$. Indeed, we need to include $u_1$ and $u_1'$ as well as $u_2$ and $u_2'$, but each of these functions is complex-valued and corresponds to two entries in the vector $\bm{x}$. Furthermore, after an appropriate change of variables, we expand each component function $x_j$ into a Chebyshev series, so that each $x_j$ now corresponds to an infinite sequence of its Chebyshev coefficients. Thus, we are effectively working with an 8-tuple of sequences of real numbers. This is done in Sections~\ref{SECT 1st order system} and \ref{SECT Cheb exp}.

\vspace{.05in}

$\bullet$ In Section~\ref{SECT NK}, we describe the classical Newton-Kantorovich theorem and its variation needed to establish our result. We first consider the limiting case $\omega=0$ in \eqref{TE u reduced} and prove the existence of a nontrivial solution satisfying \eqref{sym}. This is done in Section~\ref{SECT omega=0}, and in the preceding subsections of Section~\ref{SECT proof of main}, we present the necessary preliminaries for applying the appropriate Newton-Kantorovich theorem. In particular, the approximate solution $\bm{p}$ is computed numerically using the shooting method (cf. Figure~\ref{FIG}). The classical Newton-Kantorovich theorem requires the invertibility of $F'(\bm{p})$. Here we use a slightly modified version (cf. Theorem~\ref{THM NK2}), which allows working with an approximate inverse. Invertibility of $F'(\bm{p})$ follows as a consequence but is not needed directly. 

As a result, when dealing with norm estimates, the computations are separated into finite parts (which can be represented as vectors and matrices) and ``tails", which contain the infinite parts corresponding to the infinite tails of the Chebyshev sequences. We use analytic estimates to treat the tail parts, which are then combined with numerical estimates of the finite parts. This allows us to rigorously estimate both the norms of elements in a Banach space and the norms of operators. The numerical computations are performed on a computer using interval arithmetic \cite{Moore66, Tucker2011}, which controls rounding errors and provides rigorous verification of inequalities; for example, that an appropriate operator is a contraction and therefore has a fixed point, which then implies the existence of our desired nontrivial solution. 

\vspace{.05in}

$\bullet$ In Section~\ref{SECT omega>0}, we treat the case $\omega>0$ using a parameter continuation approach, which is again based on a version of the Newton-Kantorovich theorem requiring uniform estimates with respect to $\omega$ (cf. Theorem~\ref{THM NK3}). We obtain that, for any $\omega \in [0,2.03]$, the problem \eqref{TE u reduced} has a nontrivial solution satisfying \eqref{sym}. We are unable to continue this solution branch beyond $\omega = 2.03$ using parameter continuation, as numerical evidence indicates a bifurcation occurring between $\omega = 2.04$ and $\omega = 2.06$ (one branch consists of the conjugate solution and returns toward $\omega=0$, while the other branch is purely real-valued). Furthermore, since multiple branches exist, pseudo-arclength continuation also fails numerically. A rigorous bifurcation analysis will require additional machinery, perhaps similar to \cite{jbjpelena_hopf}, and is left for future work.

\begin{figure}[H]
\begin{center}
\includegraphics[scale=0.43]{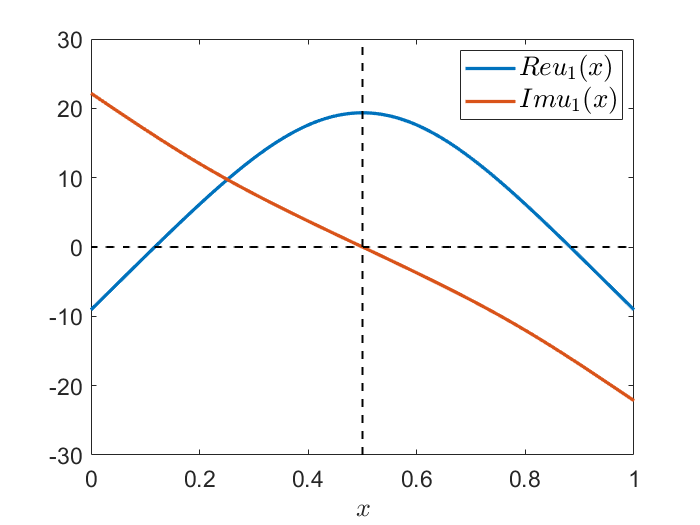}
\includegraphics[scale=0.43]{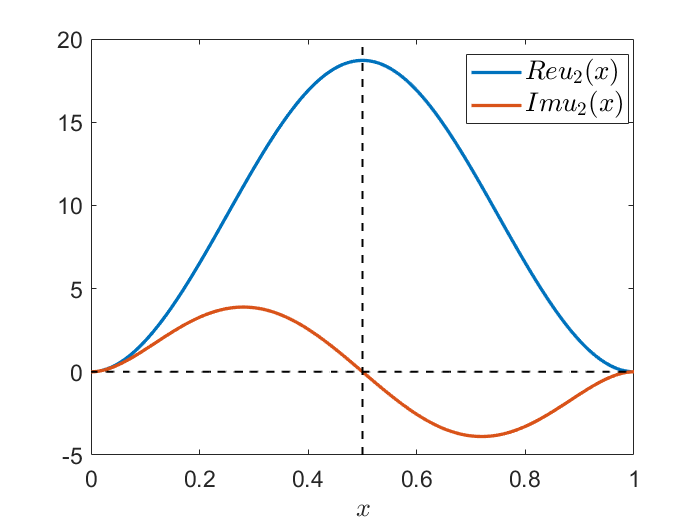}
\caption{Approximate solution of \eqref{TE u reduced} for $\omega=0$ that satisfies the symmetry \eqref{sym}.}
\label{FIG}
\end{center}
\end{figure}

We note that the approach we adopt has been widely used in the area of computer-assisted proofs \cite{Arioli2005, Ayala2026, polynomial_chaos, Day2007, Plum1992, Yamamoto2021}. We refer to \cite{G3} for a brief survey of other computer-assisted proofs for PDE problems and more particularly to \cite{G1,G2} for application to eigenvalue problems.  Finally, we want to draw attention to \cite{LiYang} in which an essential hypothesis about embedded eigenvalues for the three-dimensional cubic NLS is verified using computer-assisted techniques, as well as to the very recent work \cite{CS} in which computer-assisted techniques  have been used to construct a counterexample to the Pompeiu--Schiffer conjecture in the planar case.

\section{The Newton-Kantorovich theorem} \label{SECT NK}
\setcounter{equation}{0}

In this section, we describe in an abstract framework the functional-analytic tools that will be used to obtain the desired existence result. We begin with a brief summary of the classical Newton-Kantorovich theorem \cite{Ort68, Zeidler} and introduce a slight modification, also referred to as the radii polynomial theorem \cite{Lessard2007, Yamamoto1998}.

Suppose that $X$ and $Y$ are Banach spaces, and let $F: X \to Y$ be a nonlinear operator defined on the open ball $B_{r_0}(p)$ of radius $r_0>0$ centered at $p \in X$. Assume that $F$ is Fr{\'e}chet differentiable in this ball and that its Fr{\'e}chet derivative is Lipschitz continuous with constant $c>0$, i.e., for all $x, \tilde x\in B_{r_0}(p)$, 

\begin{equation*}
\|F'(x) - F'(\tilde x)\|_{B(X,Y)} \leq c \|x - \tilde x\|_X,
\end{equation*}

\n where the subscript $B(X,Y)$ denotes the operator norm. Assume also that the derivative $L = F'(p): X \to Y$ is an invertible linear operator (with continuous inverse). Our goal is to find a zero of $F$ near $p$. To this end, we expand

\begin{equation} \label{F expand}
F(p + h) = F(p) + L h + E(h), \qquad \qquad E(h) = \int_0^1 \left[F'(p+th) - F'(p) \right] h \, dt,
\end{equation}

\n where $E$ denotes the remainder term and we used its mean-value representation \cite{Zeidler}. Consequently, the zero-finding problem can be written as a fixed-point equation

\begin{equation*}
h = T(h), \qquad \qquad T(h) = - L^{-1} F(p) - L^{-1}E(h).
\end{equation*}

\n The goal now is to derive conditions that guarantee that $T$ maps the closed ball $\overline{B_r}(0)$ into itself and that it is a contraction on this ball (here $0<r<r_0$). The contraction mapping theorem then implies the existence of a unique fixed point $h$ in this ball, which in turn yields a solution to $F(x) = 0$ with $x = p + h$. Suppose $a$ and $b$ are positive constants such that

\begin{equation*}
\|L^{-1}\|_{B(Y,X)} \leq a, \qquad \qquad \|L^{-1} F(p)\|_X \leq b.
\end{equation*}

\n The following estimates follow directly from the integral representation of the remainder term: for all $h, h_1, h_2 \in B_{r}(0)$ 

\begin{equation*}
\|E(h)\|_Y \leq \frac{c}{2} \|h\|_X^2, \qquad \qquad \|E(h_1) - E(h_2)\|_Y \leq cr \|h_2 - h_1\|_X.
\end{equation*}

\n Consequently,   

\begin{equation*}
\|T(h)\|_X \leq \frac{ac}{2} r^2 + b, \qquad \qquad
\|T(h_1) - T(h_2)\|_X \leq acr \|h_2 - h_1\|_X.
\end{equation*}

\n We want the right-hand side of the first estimate above to be bounded by $r$. Therefore, let us introduce the quadratic

\begin{equation*}
Q(r) = \frac{ac}{2} r^2 - r + b.
\end{equation*}

\n We now state:

\begin{theorem}[Newton-Kantorovich] \label{THM NK}
Under the foregoing assumptions, suppose there exists $r \in (0,r_0)$ such that $Q(r) < 0$. Then $T$ has a unique fixed point in $B_r(0)$, or equivalently, $F$ has a unique zero in $B_r(p)$.
\end{theorem}

\begin{proof}
Since the quadratic $Q$ takes a negative value at $r\in(0,r_0)$, and since $Q(0)=b>0$, it must have two real positive roots. Consider the smaller root

\begin{equation*}
r_{-} = \frac{1 - \sqrt{1-2abc}}{ac}.
\end{equation*}

\n Clearly, $r_-<r$ and $T$ maps the ball $\overline{B}_{r_{-}}(0)$ into itself. Further, as $ac r_{-} < 1$, $T$ is a contraction in this ball. Thus, there exists a unique fixed point in this ball. In fact, the uniqueness extends to the larger ball $B_r(0)$. For details, we refer to \cite{Zeidler}.
\end{proof}

In practice, directly verifying the invertibility of $F'(p)$ may be inconvenient. To address this, we now formulate a generalization of the above theorem, which uses an operator $A$ that serves as an approximate inverse of $F'(p)$.

\begin{theorem} \label{THM NK2}
Let $X$ and $Y$ be Banach spaces, $p \in X$ and $r_0 > 0$. Let $F : X \to Y$ be Fr{\'e}chet differentiable in the ball $B_{r_0}(p)$. Assume that $A : Y \to X$ is an injective, bounded linear operator. Let $Y_0, Z_1$ and $Z_2$ be nonnegative constants such that 

\begin{enumerate}
\item[(i)] $\displaystyle \|A F (p)\|_X \leq Y_0$
\item[(ii)] $\displaystyle \|I - A F'(p) \|_{B(X)} \leq Z_1$
\item[(iii)] $\displaystyle \|A \left[ F'(x) - F'(\tilde x) \right]\|_{B(X)} \leq Z_2 \|x-\tilde x\|_X$ for all $x,\tilde x \in B_{r_0}(p)$
\end{enumerate}

\n Consider the quadratic

\begin{equation*}
Q(r) = \frac{1}{2} Z_2 r^2 - (1-Z_1) r + Y_0.
\end{equation*}

\n If there exists $r \in (0, r_0)$ such that $Q(r)<0$, then there exists a unique $x \in B_r(p)$ satisfying $F(x) = 0$. 

\end{theorem}

\begin{proof}

Since $A$ is injective, it suffices to find $x$ such that $AF(x) = 0$. Letting $x = p + h$ and expanding $F$ as in \eqref{F expand}, the zero-finding problem can then be written as the fixed-point equation

\begin{equation*}
h = T(h), \qquad \qquad T(h) = - A F(p) + \left[I - A F'(p)\right] h - AE(h).
\end{equation*}

\n The proof then proceeds in exactly the same way as that of Theorem~\ref{THM NK}.

\end{proof}

\begin{remark}
\normalfont 

Let us show that the assumptions of the above theorem imply that $F'(p): X \to Y$ is an invertible operator. Since the quadratic $Q$ takes a negative value, it must have two real roots. Moreover, the fact that $Q$ takes a negative value for some positive $r>0$ implies that both roots are positive, which in turn yields $Z_1 < 1$. Using part $(ii)$ and a Neumann series argument, it follows that the operator $A F'(p): X \to X$ is invertible. As the range of $A F'(p)$ (which is the whole space $X$) is a subset of the range of $A$, we conclude that $A: Y \to X$ is both injective and surjective, and thus invertible. Consequently, $F'(p)$ is also invertible. Note that if we choose $A = (F'(p))^{-1}$, then the above theorem reduces to Theorem~\ref{THM NK} with $Y_0 = b$, $Z_1 = 0$ and $Z_2 = a c$.
\end{remark}

\begin{remark}
\normalfont 
As already mentioned, $p$ represents an approximate solution, $F(p) \approx 0$. Thus, the key ingredients in applying the above theorem are a good approximate solution $p$ and a sufficiently accurate approximate inverse $A$ of $F'(p)$. The hypothesis $Q(r) < 0$ quantifies the required accuracy of these approximations and guarantees the existence of a true solution $x$ near $p$.

Typically, as is also the case here, the space $Y$ and its norm do not enter any of the estimates and become irrelevant for the application of this result (note that all the quantities in Theorem~\ref{THM NK3} are measured with respect to the $X$-norm). In our situation, $Y$ is a larger space with $X \subset Y$ continuously embedded. The operator $A$ (a finite-rank perturbation of a diagonal operator) is not only a bounded linear operator from $Y$ to $X$, but also from the smaller space $X$ to $X$. Furthermore, although $F: X \to Y$, the approximate solution $p \in X$ satisfies $F(p) \in X$, and the difference $F'(p+h) - F'(p)$ is also a bounded linear operator from $X$ to $X$. Because of these properties, in our estimates the space $Y$ and its norm do not appear (see Section~\ref{SECT proof of main} below).
\end{remark}

\section{Proof of Theorem~\ref{THM main 2}} \label{SECT proof of main}
\setcounter{equation}{0}

In this section we present the proof of our main result. We begin by rewriting \eqref{TE u reduced}, together with \eqref{sym}, as a first-order system. We then expand the involved functions into Chebyshev series and introduce the operator $F$ as well as the appropriate function spaces $X$ and $Y$. Next, we derive estimates that are used in the application of the Newton–Kantorovich Theorem~\ref{THM NK2} for the case $\omega = 0$. Finally, we conclude the section by studying the case $\omega > 0$, in which a different version of the Newton–Kantorovich theorem is used, together with the estimates derived in the preceding subsections.

\subsection{First-order system} \label{SECT 1st order system}

We can incorporate the symmetry \eqref{sym} into the problem \eqref{TE u reduced} by considering a reduced problem on the half-interval $\left( \frac{1}{2}, 1 \right)$ and imposing appropriate boundary conditions at $\frac{1}{2}$. Indeed, recall that \eqref{sym} implies that $\Re u_j$ is even and $\Im u_j$ is odd with respect to the midpoint $\frac{1}{2}$ for $j=1,2$. If $f$ is a smooth function on $(0, 1)$ that is even with respect to the midpoint, then $f' \left( \frac{1}{2} \right) = 0$. Analogously, if $f$ is a smooth odd function, then $f \left( \frac{1}{2} \right) = 0$. Therefore, we consider the problem

\begin{equation} \label{TE u half interval}
\begin{cases}
\displaystyle u_1'' + \omega^2  u_1  = - u_2 \overline{u}_1 \hspace{0.7in} &\text{in} \ \left( \frac{1}{2}, 1 \right)
\\[.05in]
\displaystyle u_2'' + 4 \omega^2 u_2   = - u_1^2  &\text{in} \ \left( \frac{1}{2}, 1 \right)
\\[.05in]
\Re u_j' \left( \tfrac{1}{2} \right) = \Im u_j \left( \tfrac{1}{2} \right) = 0 & j=1,2
\\[.05in]
u_2(1) =  u_2'(1)  = 0.
\end{cases}
\end{equation}

\n It is now clear that if $u_1, u_2$ solve \eqref{TE u half interval}, then extending $\Re u_j$ to the interval $(0,1)$ as an even function with respect to $\tfrac{1}{2}$, and $\Im u_j$ as an odd function, for $j=1,2$, yields a solution pair of \eqref{TE u reduced} that satisfies \eqref{sym}. Conversely, any solution of \eqref{TE u reduced} satisfying \eqref{sym} also satisfies \eqref{TE u half interval}.

It will be convenient to map the interval $\left( \frac{1}{2}, 1 \right)$ onto $(-1,1)$. This can be achieved by the change of variables $t \mapsto 4t-3$. Abusing notation, we keep the same notation $u_j$ for the functions that are now defined on the interval $(-1,1)$. Note that this change of variables affects the derivatives by introducing a multiplicative factor of 4. Therefore, the problem \eqref{TE u half interval} now becomes

\begin{equation*}
\begin{cases}
\displaystyle 16 u_1'' + \omega^2  u_1  = - u_2 \overline{u}_1 \hspace{0.7in} &\text{in} \ \left( -1 , 1 \right)
\\[.05in]
\displaystyle 16 u_2'' + 4 \omega^2 u_2   = - u_1^2  &\text{in} \ \left( -1, 1 \right)
\\[.05in]
\Re u_j' \left( -1 \right) = \Im u_j \left( -1 \right) = 0 & j=1,2
\\[.05in]
u_2(1) =  u_2'(1)  = 0.
\end{cases}
\end{equation*}

\n To rewrite the above problem as a first-order system, we introduce the notation

\begin{equation*}
\begin{split}
u_1 = \xx_1 + i \xx_2, \qquad \qquad 4u_1'= \xx_3 + i \xx_4
\\
u_2 = \xx_5 + i \xx_6, \qquad \qquad 4u_2'= \xx_7 + i \xx_8
\end{split}
\end{equation*}

\n and setting $\bm{\xx} = (\xx_1,...,\xx_8)$ we arrive at the first-order system

\begin{equation} \label{1st order system}
\begin{cases}
\displaystyle \bm{\xx}' = \bm{\ff}(\bm{\xx}) \hspace{0.7in} &\text{in} \ \left( -1 , 1 \right)
\\[.05in]
\xx_j(-1) = 0 & j = 2, 3, 6, 7
\\[.05in]
\xx_j(1) = 0 & j = 5, 6, 7, 8,
\end{cases}
\end{equation}

\n where, writing $\bm{\ff}$ as a column vector for convenience, we have

\begin{equation} \label{f}
\bm{\ff}(\bm{\xx}) = 
\begin{bmatrix}
\ff_1(\bm{\xx}) \\ \ff_2(\bm{\xx}) \\ \ff_3(\bm{\xx}) \\ \ff_4(\bm{\xx}) \\ \ff_5(\bm{\xx}) \\ \ff_6(\bm{\xx}) \\ \ff_7(\bm{\xx}) \\ \ff_8(\bm{\xx}) \\
\end{bmatrix}
=
\frac{1}{4} 
\begin{bmatrix} 
\xx_3 \\ \xx_4 \\
-\xx_1 \xx_5 - \xx_2 \xx_6 - \omega^2 \xx_1 \\
\xx_2 \xx_5 - \xx_1 \xx_6 - \omega^2 \xx_2 \\
\xx_7 \\ \xx_8 \\
-\xx_1^2 + \xx_2^2 - 4\omega^2 \xx_5 \\ 
-2\xx_1 \xx_2 - 4\omega^2 \xx_6
\end{bmatrix}.
\end{equation}

\subsection{Chebyshev expansions and the operator $F$} \label{SECT Cheb exp}

Our goal is to expand the functions $\xx_j(t)$, for $j=1,...,8$, in terms of Chebyshev polynomials $T_n(t)$ of the first kind and order $n$, and to rewrite the system \eqref{1st order system} by regarding the Chebyshev coefficients of these expansions as our unknowns. Approximating functions by Chebyshev polynomials is a classic approach, and in fact any Lipschitz continuous function on $[-1,1]$ has a Chebyshev series expansion that converges uniformly and absolutely \cite{Trefethen2013}. Thus, we write

\begin{equation} \label{Cheb expansion}
\xx_j(t) = x_{j,0} + 2 \sum_{n=1}^\infty x_{j,n} T_n(t),
\end{equation}

\n where $x_j = (x_{j,n})_{n \geq 0}$ denotes the sequence of Chebyshev coefficients, and we use the boldface notation $\bm{x}=(x_1,...,x_8)$ to denote the 8-tuple of sequences of Chebyshev coefficients corresponding to the vector function $\bm{\xx}(t)$.

\vspace{.1in}

\n \textbf{Notation:} To avoid any confusion, we use the font $\xx$ (similarly $\ff$, etc.) to denote functions, while $x$ (resp. $f$, etc.) denotes the corresponding sequence of Chebyshev coefficients. Further, given a sequence $x_j$ we use $x_{j,n}$ to denote its $n$-th term. At times, we also use the alternative notation $[x_j]_n$ to denote the $n$-th term. Finally, $[\bm{x}]_n$ denotes the 8-tuple consisting of the $n$-th terms of its component sequences.

\vspace{.1in}

Continuing, we expand the right-hand sides in \eqref{f} (which are quadratic expressions in terms of $\xx_j$) into their respective Chebyshev series:  

\begin{equation} \label{f_j Cheb expansion}
\ff_j(\bm{\xx}(t)) = f_{j,0}(\bm{x}) + 2 \sum_{n=1}^\infty f_{j,n}(\bm{x}) T_n(t).
\end{equation}

\n Here the notation $f_{j,n}(\bm{x})$ indicates the dependence of these coefficients on the Chebyshev coefficients $\bm{x}$. We are going to integrate the ODE in \eqref{1st order system}. To that end, we will need the following:

\begin{lemma}
Let $\ff(t) = f_0 +  2 \sum_{n=1}^\infty f_n T_n(t)$, then

\begin{equation*}
\int \ff(t) dt = \sum_{n=1}^\infty \frac{f_{n-1}-f_{n+1}}{n} T_n(t) + \text{const}.
\end{equation*}

\end{lemma}

\begin{proof}
Integrating the recurrence relation $2 T_n = \frac{1}{n+1} T_{n+1}' - \frac{1}{n-1} T_{n-1}'$ for $n\geq 2$ and using that $T_0(t)=1$ and $T_1(t)=t$, we obtain (suppressing the additive constant of integration in the notation)

\begin{equation*}
\int \ff(t) dt = f_0 t + f_1 t^2 + \sum_{n=2}^\infty f_n \left( \frac{T_{n+1}}{n+1} - \frac{T_{n-1}}{n-1} \right).  
\end{equation*}

\n In the above formula we then replace $t = T_1$ and $t^2 = \frac{1+T_2}{2}$. We further separate the sum into two parts and reindex them to obtain

\begin{equation*}
\int \ff(t) dt = \frac{f_1}{2} + f_0 T_1 + \frac{f_1}{2} T_2 + \sum_{k=3}^\infty \frac{f_{k-1}}{k} T_k - \sum_{k=1}^\infty \frac{f_{k+1}}{k} T_k 
\end{equation*}

\n The first term can be ignored, as it is a constant and can be combined with the additive constant of integration. The second and third terms can be included in the first sum if we adjust the summation index to start from $k=1$. We then combine this with the second sum and conclude the proof.
\end{proof}

\n Let us now integrate the differential equation $\xx_j'(t) = \ff_j(\bm{\xx}(t))$ using the Chebyshev expansion \eqref{f_j Cheb expansion} and the above lemma to obtain, for $j=1,...,8$,

\begin{equation} \label{x_j,n relation}
2x_{j,n} = \frac{f_{j, n-1}(\bm{x}) - f_{j, n+1}(\bm{x})}{n}, \qquad \qquad n=1,2,...
\end{equation}

\n Note that there is no condition on the zeroth-order coefficient $x_{j,0}$ because of the free constant of integration. Conditions on this coefficient arise through the boundary conditions \eqref{1st order system}. In view of the formulas $T_n(-1) = (-1)^n$ and $T_n(1)=1$, we can rewrite these boundary conditions as

\begin{equation} \label{BC x_j,n}
\begin{cases}
\displaystyle x_{j,0} + 2 \sum_{n=1}^\infty (-1)^n x_{j,n} = 0, \qquad \qquad &j = 2, 3, 6, 7
\\[.2in]
\displaystyle x_{j,0} + 2 \sum_{n=1}^\infty x_{j,n} = 0, &j = 5, 6, 7, 8.
\end{cases}
\end{equation}

\n Thus, the system \eqref{1st order system} can be rewritten as \eqref{x_j,n relation} and \eqref{BC x_j,n} in terms of the Chebyshev coefficients. Let us now introduce $\bm{F} = (F_1,...,F_8)$ so that the system can be conveniently written as $\bm{F}(\bm{x}) = 0$. Each $F_j$ is a sequence with terms $F_{j,n}$. The relation \eqref{x_j,n relation} can be written as $F_{j,n}(\bm{x}) = 0$ for $n\geq 1$, while the boundary conditions \eqref{BC x_j,n} are incorporated at index $n=0$ and rewritten as $F_{j,0}(\bm{x}) = 0$. To rewrite the boundary conditions more conveniently, let us introduce the sequences

\begin{equation} \label{alpha beta}
\alpha_n = 
\begin{cases}
1, \ n=0
\\
2, \ n \geq 1
\end{cases},
\qquad \qquad
\beta_n = (-1)^n \alpha_n
\end{equation}

\n and define

\begin{equation} \label{B 1-4}
B_1\bm{x} = \sum_{n=0}^\infty \beta_n x_{2,n}, \quad
B_2\bm{x} = \sum_{n=0}^\infty \beta_n x_{3,n}, \quad
B_3\bm{x} = \sum_{n=0}^\infty \beta_n x_{6,n}, \quad
B_4\bm{x} = \sum_{n=0}^\infty \beta_n x_{7,n},
\end{equation}

\n which correspond to the first four conditions in \eqref{BC x_j,n} and for the next four conditions set

\begin{equation} \label{B 5-8}
B_j\bm{x} = \sum_{n=0}^\infty \alpha_n x_{j,n}, \qquad \qquad j=5, 6, 7, 8.
\end{equation}

\n In summary, \eqref{x_j,n relation} and \eqref{BC x_j,n} can be rewritten as 

\begin{equation} \label{F = 0}
\bm{F}(\bm{x}) = 0,
\end{equation}

\n where, for $j=1,..,8$,

\begin{equation} \label{F_j,n}
F_{j,n}(\bm{x}) = 
\begin{cases}
B_j\bm{x}, \hspace{.5in} &n=0
\\
2n x_{j,n} - f_{j,n-1}(\bm{x}) + f_{j,n+1}(\bm{x}), & n \geq 1
\end{cases}
\end{equation}

\vspace{.1in}

Let us next rewrite the above equations in a more convenient vector form. To that end, we first introduce the linear operators $M$ and $S$, defined on the space of sequences (for the concrete function spaces, see the next section), as follows: for any sequence $a=(a_n)_{n \geq 0}$,

\begin{equation} \label{M S}
[Ma]_n = 2 n a_n, \qquad n \geq 0 \qquad \text{and} \qquad
[S a]_n = - a_{n-1} + a_{n+1}, \qquad n \geq 1.
\end{equation}

\n Here recall that square brackets with subscript $n$ denote the $n$-th term of the corresponding sequence, and we also set $[S a]_0 = 0$. The second equation of \eqref{F_j,n} now becomes $F_{j,n}(\bm{x}) = [M x_j + S f_j(\bm{x})]_n$ for all $n \geq 1$. Let us further set 

\begin{equation*}
\bm{B} \bm{x} = (B_1 \bm{x},..., B_8 \bm{x}), \qquad \bm{f} (\bm{x}) = \left(f_1 (\bm{x}),..., f_8 (\bm{x}) \right)
\end{equation*}

\n and

\begin{equation*}
\bm{M} \bm{x} = (M x_1,...,M x_8), \qquad \bm{S} \bm{x} = (S x_1,...,S x_8).
\end{equation*}

\n Then we can rewrite \eqref{F_j,n} in the vector form as

\begin{equation} \label{F M S vector}
\bm{F}(\bm{x}) = 
\begin{cases}
\bm{B} \bm{x} \hspace{.5in} &n=0
\\
\bm{M} \bm{x} + \bm{S} \bm{f}(\bm{x}), & n\geq 1
\end{cases}
\end{equation}

\n where we have dropped $[\cdot]_n$ from either side of the equation for notational convenience.

\subsection{Function spaces} \label{SECT function spaces}

\n Let us now introduce the appropriate function spaces for analyzing the equation \eqref{F = 0}. For $\nu \geq 1$, we define the Banach space of weighted $\ell^1$ sequences

\begin{equation*}
\ell^1_\nu = \left\{ a = (a_n)_{n\geq 0} \ : \ \|a\|_{\ell^1_\nu} < \infty \right\}, \qquad \qquad \|a\|_{\ell^1_\nu} = |a_0| + 2 \sum_{n=1}^\infty |a_n| \nu^n.
\end{equation*}

\n Note that larger values of $\nu$ correspond to faster decay of the sequence $a_n$, and as a result, the corresponding function whose Chebyshev coefficients are $a_n$ is more regular. For instance, if a function extends to a complex analytic function in an ellipse with foci at $\pm 1$, then its Chebyshev coefficients decay exponentially \cite{Trefethen2013}. It is clear that $S : \ell^1_\nu \to \ell^1_\nu$ is a bounded operator. However, if $a \in \ell^1_\nu$, the sequence $(n a_n)_{n \geq 0}$ does not, in general, lie in this space. In other words, the operator $M$ does not map $\ell^1_\nu$ to itself. Therefore, we also introduce the larger space

\begin{equation*}
\tilde{\ell}^1_\nu = \left\{ b = (b_n)_{n\geq 0} \ : \ \|b\|_{\tilde{\ell}^1_\nu} < \infty \right\}, \qquad \qquad \|b\|_{\tilde{\ell}^1_\nu} = |b_0| + 2 \sum_{n=1}^\infty\frac{|b_n|}{n} \nu^n.
\end{equation*}

\n It is now evident that $M: \ell^1_\nu \to \tilde{\ell}^1_\nu$ is a bounded operator. We next introduce the product Banach spaces

\begin{equation} \label{X}
X = \left( \ell^1_\nu \right)^8, \qquad \qquad \|\bm{x}\|_X = \max_{j=1,...,8} \|x_j\|_{\ell^1_\nu}
\end{equation}

\n and 

\begin{equation} \label{Y}
Y = \left( \tilde{\ell}^1_\nu \right)^8, \qquad \qquad \|\bm{y}\|_Y = \max_{j=1,...,8} \|y_j\|_{\tilde{\ell}^1_\nu}.
\end{equation}

\begin{lemma} \label{THM F well defined}

Let the nonlinear operators $\bm{f}$ and $\bm{F}$ be as in \eqref{f} and \eqref{F M S vector}, then

\begin{enumerate}
\item[(i)] $\bm{f} : X \to X$ is defined on all of $X$ and is Fr{\'e}chet differentiable there.

\item[(ii)] $\bm{F} : X \to Y$ is defined on all of $X$ and is Fr{\'e}chet differentiable there.
\end{enumerate}

\end{lemma}

Note that part $(ii)$ directly follows from part $(i)$. Indeed, since $X \subset Y$ and this embedding is continuous, the image of $\bm{f}$ also lies in $Y$. Further, $\bm{M} : X \to Y$ and $\bm{S} : X \to X$ are bounded linear operators. Let us show that $\bm{f}$ is well defined on all of $X$ and maps this space into itself. Establishing Fr{\'e}chet differentiability is straightforward, and we omit the details (see Lemma~\ref{LEM f} and \eqref{F'p} for the formula of the Fr{\'e}chet derivative). Since the components $f_j$ of $\bm{f}$ correspond to Chebyshev coefficients of products of two functions -- i.e., the nonlinearities in \eqref{f} are quadratic -- the desired result follows from the following Banach algebra property of the space $\ell^1_\nu$ with respect to Chebyshev products, namely part $(i)$ of the lemma below (part $(ii)$ of the lemma is used later, in part $(ii)$ of Lemma~\ref{LEM f}).

\begin{lemma} \label{LEM Cheb product}
Let $\xx(t), \yy(t)$ be two functions with Chebyshev coefficients $x=(x_n)_{n\geq 0}$ and $y=(y_n)_{n\geq 0}$, respectively. Let $z=(z_n)_{n \geq 0}$ denote the Chebyshev coefficients of the product $\zz(t) = \xx(t) \yy(t)$, then the following hold true:

\begin{enumerate}
\item[(i)] $\displaystyle \|z\|_{\ell^1_\nu} \leq \|x\|_{\ell^1_\nu} \|y\|_{\ell^1_\nu}$

\item[(ii)] Let $N \geq 1$ be an integer and assume that 

\begin{equation*}
\begin{cases}
x_n = 0, \quad &n \geq N
\\
y_n = 0, & n \leq 2N-1
\end{cases},
\qquad \text{then} \qquad
z_n = 0, \quad n \leq N.
\end{equation*}
\end{enumerate}

\end{lemma}

\begin{proof}
Let us set $\tilde{x}_0 = x_0$ and $\tilde{x}_n = 2 x_n$ for $n\geq 1$, so that $\xx(t) = \sum_{n\geq 0} \tilde{x}_n T_n(t)$. Similarly define $\tilde{y}_n$ and $\tilde{z}_n$. In view of the relation

\begin{equation*}
2 T_n T_m = T_{n+m} +T_{|n-m|},
\end{equation*}

\n which holds for all $n,m \geq 0$, we obtain that for all $k \geq 0$

\begin{equation} \label{product}
\tilde{z}_k = \frac{1}{2} \sum_{n,m = 0}^\infty \tilde{x}_n \tilde{y}_m \left( \delta_{n+m, k} + \delta_{|n-m|, k} \right).
\end{equation}

\n Here $\delta$ denotes the Kronecker delta. Consequently,

\begin{equation*}
\|z\|_{\ell^1_\nu} = \sum_{k = 0}^\infty |\tilde{z}_k| \nu^k \leq \frac{1}{2} \sum_{n,m = 0}^\infty |\tilde{x}_n| \, |\tilde{y}_m| \left( \nu^{n+m} + \nu^{|n-m|} \right).   
\end{equation*}

\n To conclude the proof of part $(i)$, it remains to use the estimate $\nu^{|n-m|} \leq \nu^{n+m}$ and separate the double sums into a product of two sums using $\nu^{n+m} = \nu^n \nu^m$.

Let us now turn to part $(ii)$. By assumption, in the sum \eqref{product} we must have $n \leq N-1$ and $m \geq 2N$, otherwise the corresponding terms in the sum are zero. In particular, $m+n \geq 2N$ and $|m-n| \geq N+1$, and consequently, $\tilde{z}_k = 0$ for all $k \leq N$.

\end{proof}

\subsection{Application of Theorem~\ref{THM NK2}}

Our goal is to apply the Newton–Kantorovich Theorem~\ref{THM NK2} to the operator $\bm{F}$ defined in \eqref{F_j,n}, acting between the spaces $X$ and $Y$ as introduced in the previous section.

\subsubsection{The point $p$ and the operator $A$} \label{SECT p and A}

To start, we need an approximate solution of the equation \eqref{F = 0}, namely, a point $\bm{p} \in X$ such that $\bm{F}(\bm{p}) \approx 0$. This is obtained numerically (see Section~\ref{SECT omega=0}). In particular, $\bm{p} = (p_1,...,p_8)$ is a finite sequence -- more precisely, an 8-tuple of finite sequences. For the purposes of this part, this is the only information we need about the point $\bm{p}$. Given an integer $N \geq 1$, we introduce the following truncation operators: for any sequence $a \in \ell_1^\nu$,

\begin{equation*}
[\pi^N a]_n = 
\begin{cases}
a_n, \qquad &n=0,...,N
\\
0, &n>N
\end{cases}
\qquad \qquad
[\pi_N a]_n = 
\begin{cases}
0, \qquad &n=0,...,N
\\
a_n, &n>N.
\end{cases}
\end{equation*}

\n We extend these operators to the product space $X$ in the usual way: $\bm{\pi}_N \bm{p} = (\pi_N p_1,...,\pi_N p_8)$, and similarly for $\bm{\pi}^N$. Next, we introduce the operator $A : Y \to X$, which serves as an approximate inverse to $\bm{F}'(\bm{p})$. Differentiating the equation \eqref{F M S vector} and using the fact that $\bm{B}$, $\bm{M}$ and $\bm{S}$ are linear operators, we obtain

\begin{equation} \label{F'p}
\bm{F}'(\bm{p}) = 
\begin{cases}
\bm{B} \hspace{.5in} &n=0
\\
\bm{M} + \bm{S} \bm{f}'(\bm{p}), & n\geq 1.
\end{cases}
\end{equation}

\begin{remark}
\normalfont To be more precise, the second equation of \eqref{F'p} means that for all $j=1,...,8$ and $\bm{w} \in X$,

\begin{equation*}
\left[ \bm{F}'(\bm{p}) \bm{w} \right]_{j, n} = \left[ \bm{M} \bm{w} + \bm{S} \bm{f}'(\bm{p}) \bm{w} \right]_{j,n} \qquad \qquad n \geq 1.
\end{equation*}

\n Here, $[\cdot]_{j, n}$ denotes the $n$-th term of the $j$-th sequence in a given 8-tuple of sequences.
\end{remark}

\vspace{.1in}

We can write $\bm{F}'(\bm{p}) = \bm{\pi}^N \bm{F}'(\bm{p}) + \bm{\pi}_N \bm{F}'(\bm{p})$. Since $N \geq 1$, in the second term we can use the second equation of \eqref{F'p}, and the first term we can split again using the projection operators to arrive at

\begin{equation} \label{F'(p) expansion}
\bm{F}'(\bm{p}) = \bm{\pi}^N \bm{F}'(\bm{p}) \bm{\pi}^N + \bm{\pi}_N \bm{M} + \bm{\pi}^N \bm{F}'(\bm{p}) \bm{\pi}_N + \bm{\pi}_N \bm{S} \bm{f}'(\bm{p}). 
\end{equation}

\n Now, the first term above is the ``finite part" of the operator $\bm{F}'(\bm{p})$: it takes finite (8-tuples of) sequences as input and returns finite (8-tuples of) sequences. In our later calculations, when we numerically obtain $\bm{p}$, we can then numerically invert this operator, since it can be represented as a matrix. Note that the operator $\bm{A}$ in Theorem~\ref{THM NK2} must be injective. Therefore, working only with the finite-part operator is not sufficient, and we must combine the finite part with a ``tail" operator to achieve injectivity. The second operator in \eqref{F'(p) expansion} defines a tail operator, whose nonzero terms start from index $n=N+1$. In particular, there is no intersection with the finite-part operator. Thus, as the operator $\bm{A}$, we take the inverse of $\bm{\pi}^N \bm{F}'(\bm{p}) \bm{\pi}^N + \bm{\pi}_N \bm{M}$. More precisely, we now summarize the assumptions on $\bm{p}$ and $\bm{A}$ that will be used to derive all the estimates for the application of Theorem~\ref{THM NK2}.

\begin{enumerate}
\item[(H1)] $\bm{p} \in X$ is a finite sequence of size $N$, i.e. $\bm{\pi}_N \bm{p} = 0$

\item[(H2)] $\bm{A} = \bm{A}^N + \bm{A}_N$, where 

\begin{enumerate}
\item[(a)] $\bm{A}^N$ is a ``finite" operator, i.e.

\begin{equation} \label{A^N pi}
\bm{\pi}_N \bm{A}^N = \bm{A}^N \bm{\pi}_N = 0
\end{equation}

\n and it is injective in the space $\bm{\pi}^N X$ of (8-tuples of) finite sequences. 

\vspace{.1in}

\item[(b)] $\displaystyle [A_N a]_n =
\begin{cases}
0, \qquad &n=0,...,N
\\[.1in]
\dfrac{a_n}{2n}, & n>N
\end{cases}$  for any sequence $a$ and $\bm{A}_N$ is the extension of 

\vspace{.05in}

$A_N$ to the product space $X$, obtained by applying it componentwise.
\end{enumerate}

\end{enumerate}

\begin{remark}
\normalfont Note that the operator $\bm{\pi}^N$ (resp. $\bm{\pi}_N$) can be applied from the left as well as from the right of $\bm{A}^N$ (resp. $\bm{A}_N$) without affecting it. We will use this property in the calculations below. Further, $\bm{A}_N$ is obtained by inverting the tail of the operator $\bm{M}$:

\begin{equation} \label{BM=Id}
\bm{A}_N \bm{M} = \bm{\pi}_N I.
\end{equation} 
\end{remark}

\vspace{.1in}

Note that $\bm{A}: X \to X$ is an injective, bounded operator and

\begin{equation} \label{A_N bound}
\|\bm{A}_N\|_{B(X)} \leq \frac{1}{2(N+1)}.
\end{equation}

\n We are going to derive estimates for the quantities appearing in parts $(i)$–$(iii)$ of the Newton–Kantorovich Theorem~\ref{THM NK2} under the hypotheses (H1) and (H2). In other words, we will obtain expressions for the constants $Y_0$, $Z_1$ and $Z_2$ in terms of ``finite" quantities that can be numerically evaluated by choosing $\bm{p}$ such that $F(\bm{p}) \approx 0$ and by taking $\bm{A}^N$ to be a numerical inverse of the finite part $\bm{\pi}^N F'(\bm{p}) \bm{\pi}^N$, which in particular satisfies assumption (H2). We will then verify the hypothesis on $Q$ in Theorem~\ref{THM NK2}.

\subsubsection{Two auxiliary lemmas}

\begin{lemma} \label{LEM S bound}
Let $\bm{S} : X \to X$ be the operator defined by \eqref{M S}, then

\begin{equation*}
\|\bm{S}\|_{B(X)} \leq 2\nu.
\end{equation*}

\end{lemma}

\begin{proof}
We remark that this operator was also used in \cite{Azpeitia_2020}, which contains the proof of the above estimate. In fact, the proof is a simple application of the triangle inequality, therefore we present it here as well to keep the exposition self-contained. For any $a \in \ell^1_\nu$,

\begin{equation*}
\begin{split}
\|Sa\|_{\ell^1_\nu} &= 2 \sum_{n=1}^\infty |a_{n+1} - a_{n-1}| \nu^n \leq \frac{2}{\nu} \sum_{n=1}^\infty |a_{n+1}| \nu^{n+1} + 2\nu \sum_{n=1}^\infty |a_{n-1}| \nu^{n-1} = 
\\
&= \frac{1}{\nu} \left( \|a\|_{\ell^1_\nu} - |a_0| - 2\nu |a_1|  \right) + \nu \left( \|a\|_{\ell^1_\nu} + |a_0| \right) \leq \left( \nu + \frac{1}{\nu} \right) \|a\|_{\ell^1_\nu} + \left( \nu - \frac{1}{\nu} \right) |a_0|.
\end{split}
\end{equation*}

\n Bounding $|a_0| \leq \|a\|_{\ell^1_\nu}$, the last expression simplifies to $2\nu \|a\|_{\ell^1_\nu}$.
 
\end{proof}

\begin{lemma} \label{LEM f}
Let $\bm{f}: X \to X$ be as in \eqref{f}. 

\begin{enumerate}
\item[(i)] The Fr{\'e}chet derivative $\bm{f}'$ is Lipschitz continuous: for any $\bm{p}, \bm{h} \in X$

\begin{equation} \label{f' Lip}
\|\bm{f}'(\bm{p}+\bm{h}) - \bm{f}'(\bm{p})\|_{B(X)} \leq \|\bm{h}\|_X
\end{equation}  

\item[(ii)] Let $N\geq 1$ and assume $\bm{p}$ satisfies (H1), then $\displaystyle \bm{\pi}^{N+1} \bm{f}'(\bm{p}) \bm{\pi}_{2N+1} = 0$

\item[(iii)] Let $\bm{p}=(p_1,...,p_8) \in X$, then $\displaystyle \|\bm{f}'(\bm{p})\|_{B(X)} \leq C$, where

\begin{equation} \label{C}
C = \frac{1}{4} \max\left\{1; \ \|p_5 \pm \omega^2 e_0\|_{\ell^1_\nu} + \sum_{j=1,2,6} \|p_j\|_{\ell^1_\nu}; \ 4\omega^2 + 2 \sum_{j=1,2} \|p_j\|_{\ell^1_\nu} \right\}
\end{equation}

\n and the sequence $e_0=(e_{0,n})_{n\geq 0}$ is defined by $e_{0,0}=1$ and $e_{0,n}=0$ for $n\geq 1$.

\end{enumerate}

\end{lemma}

\begin{remark}
\normalfont In parts $(i)$ and $(iii)$, the point $\bm{p}$ is arbitrary and does not need to satisfy the assumption (H1). Further, it does not appear on the right-hand side of the Lipschitz estimate \eqref{f' Lip}, since the nonlinearities are quadratic (cf. \eqref{f}). This Lipschitz estimate is used in part $(ii)$ of Lemma~\ref{LEM Y0 and Z_2}, in the $Z_2$ bound. Parts $(ii)$ and $(iii)$ of the above lemma are used in Lemma~\ref{LEM Z_1}, in the $Z_1$ bound.
\end{remark}

\begin{proof}

Let $\bm{p}=(p_1,...,p_8) \in X$, and with our notational convention from Section~\ref{SECT Cheb exp}, let $\bm{\pp} = (\pp_1,...,\pp_8)$, where $\pp_j(t)$ denotes the function whose Chebyshev coefficients are given by the sequence $p_j$. Similarly, we use $\bm{\hh}$ and $\bm{\ww}$ to denote the 8-tuples of functions whose corresponding sequences of Chebyshev coefficients are given by $\bm{h}, \bm{w} \in X$, respectively. Finally, recall that $\bm{f}=(f_1,...,f_8)$, where $f_j : X \to \ell^1_\nu$, and that $f_j(\bm{p})$ denotes the sequence of Chebyshev coefficients of the function $\ff_j(\bm{\pp})$ defined by \eqref{f}. Let us now compute the Fr{\'e}chet derivative of $\bm{f}$. First, from \eqref{f} we have $\ff_1(\bm{\xx}) = \frac{1}{4} \xx_3$, and therefore $f_1(\bm{x}) = \frac{1}{4} x_3$. Similarly, $f_2$, $f_5$ and $f_6$ are also linear in $\bm{x}$, and consequently their Fr{\'e}chet derivatives are equal to themselves and do not depend on the point $\bm{p}$, namely

\begin{equation*}
f_1'(\bm{p}) \bm{w} = \frac{1}{4} w_3, \qquad f_2'(\bm{p}) \bm{w} = \frac{1}{4} w_4, \qquad
f_5'(\bm{p}) \bm{w} = \frac{1}{4} w_7, \qquad f_6'(\bm{p}) \bm{w} = \frac{1}{4} w_8.
\end{equation*}

\n The nonlinearities in \eqref{f} appear in the components $j=3,4,7,8$. The Fr{\'e}chet derivatives of these components are given by the following formulas: 

\begin{equation*}
\begin{split}
f_j'(\bm{p})\bm{w} \ &\text{is the sequence of Chebyshev coefficients of}
\\
&\text{the function }
\frac{1}{4}
\begin{cases}
-(\pp_5 + \omega^2) \ww_1 - \pp_1 \ww_5 - \pp_6 \ww_2 - \pp_2 \ww_6, \qquad &j=3
\\
(\pp_5 - \omega^2) \ww_2 + \pp_2 \ww_5 - \pp_6 \ww_1 - \pp_1 \ww_6, &j=4
\\
- 4\omega^2 \ww_5 + 2 \left(\pp_2 \ww_2 - \pp_1 \ww_1 \right), &j=7
\\
- 4\omega^2 \ww_6  - 2 ( \pp_1 \ww_2 + \pp_2 \ww_1),  &j=8
\end{cases}
\end{split}
\end{equation*}

\n Let us show this result for $j=8$, the argument is completely analogous for the other components. By definition $\ff_8 (\bm{\xx}) = - \omega^2 \xx_6 - \frac{1}{2} \xx_1 \xx_2 $, therefore

\begin{equation*}
\ff_8(\bm{\pp}+\bm{\ww}) - \ff_8(\bm{\pp})  = - \omega^2 \ww_6  - \frac{1}{2} (\pp_1 \ww_2 + \pp_2 \ww_1) - \frac{1}{2} \ww_1 \ww_2.
\end{equation*}

\n Dropping the last term, we obtain the linear part with respect to $\bm{\ww}$, which corresponds exactly to $f_8'(\bm{p})\bm{w}$.

\vspace{.1in}

$(i)$ We need to show that for all $j=1,...,8$ and any $\bm{w} \in X$  

\begin{equation*}
\|f_j'(\bm{p}+\bm{h})\bm{w} - f_j'(\bm{p})\bm{w}\|_{\ell^1_\nu} \leq \|\bm{h}\|_X \|\bm{w}\|_X.
\end{equation*}

\n The above estimate is trivial for $j=1,2,5,6$, since the left-hand side is zero in those cases. Let us prove the result for $j=8$ (the argument is analogous for the remaining values of $j$). Note that $f_8'(\bm{p}+\bm{h})\bm{w} - f_8'(\bm{p})\bm{w}$ is equal to the sequence of Chebyshev coefficients of the function

\begin{equation*}
-\frac{1}{2} \left( \hh_1 \ww_2 + \hh_2 \ww_1 \right).
\end{equation*}

\n Using Lemma~\ref{LEM Cheb product} we can bound the $\ell^1_\nu$-norm of the Chebyshev sequence of the product of two functions by the product of the norms of the individual Chebyshev sequences, so that

\begin{equation*}
\|f_8'(\bm{p}+\bm{h})\bm{w} - f_8'(\bm{p})\bm{w}\|_{\ell^1_\nu} \leq \frac{1}{2} \left( \|h_1\|_{\ell^1_\nu} \|w_2\|_{\ell^1_\nu} + \|h_2\|_{\ell^1_\nu} \|w_1\|_{\ell^1_\nu} \right) \leq \|\bm{h}\|_X \|\bm{w}\|_X.
\end{equation*}

\vspace{.1in}

$(ii)$ Let $\bm{w} \in X$ be such that $\bm{w} = \bm{\pi}_{2N+1} \bm{w}$, in other words $\left[ w_j \right]_n = 0$ for $n=0,...,2N+1$ and all $j=1,...,8$. Our goal is to prove that for all $j$,

\begin{equation} \label{0}
\left[f_j'(\bm{p}) \bm{w}\right]_n = 0, \qquad \qquad n=0,...,N+1.
\end{equation}

\n Recall that in this part we also assume $\bm{\pi}_N \bm{p} = 0$, i.e. $[p_j]_n = 0$ for $n\geq N+1$. Our previous calculations show that $f_j'(\bm{p}) \bm{w}$ consists of two types of terms:

\begin{enumerate}
\item[$\bullet$] a linear term -- $w_l$ for some index $l$. For this term \eqref{0} holds trivially.

\item[$\bullet$] a quadratic term -- the sequence of Chebyshev coefficients of the product $\pp_l \ww_s$ for some indices $l, s$. For this term, \eqref{0} follows directly from part $(ii)$ of Lemma~\ref{LEM Cheb product} (applied with $N+1$ in place of $N$).
\end{enumerate}

$(iii)$ The estimate follows directly from the formulas for $f_j'(\bm{p})\bm{w}$ obtained above and from Lemma~\ref{LEM Cheb product}, which bounds the Chebyshev coefficients of the product of two functions. We simply observe that for $j=3,4$, the function $\pp_5 \pm \omega^2$ has Chebyshev coefficients given by the sequence $p_5 \pm \omega^2 e_0$.
\end{proof}

\subsubsection{The $Y_0$ and $Z_2$ bounds}

\n In this section, we derive bounds for the quantities $(i)$ and $(iii)$ in Theorem~\ref{THM NK2} and obtain expressions for the constants $Y_0$ and $Z_2$.

\begin{lemma} \label{LEM Y0 and Z_2}
Assume (H1) and (H2). Then

\begin{enumerate}
\item[(i)] $\displaystyle \|\bm{A} \bm{F} (\bm{p})\|_X \leq Y_0$, where

\begin{equation*}
Y_0 = \|\bm{A}^N \bm{F} (\bm{p})\|_X + \frac{1}{2(N+1)} \|\left(\bm{\pi}^{2N+1} - \bm{\pi}^{N}\right) \bm{S} \bm{f}(\bm{p})\|_X.
\end{equation*}

\item[(ii)] $\displaystyle \|\bm{A} \left[ \bm{F}'(\bm{p}+\bm{h}) - \bm{F}'(\bm{p} +\bm{\tilde h}) \right]\|_{B(X)} \leq Z_2 \|\bm{h}-\bm{\tilde h}\|_X$ for all $\bm{h},\bm{\tilde h} \in X$, where

\begin{equation*}
Z_2 = 2\nu \left( \|\bm{A}^N\|_{B(X)} + \frac{1}{2(N+1)} \right).
\end{equation*}

\end{enumerate}

\end{lemma}

\begin{proof}

$(i)$ By definition $\bm{A} \bm{F}(\bm{p}) = \bm{A}^N \bm{F}(\bm{p}) + \bm{A}_N \bm{F}(\bm{p})$. In view of \eqref{F M S vector}, 

\begin{equation*}
\bm{A}_N \bm{F}(\bm{p}) = \bm{A}_N \bm{M}\bm{p} + \bm{A}_N \bm{S} \bm{f}(\bm{p}) = \bm{A}_N \bm{\pi}_N \bm{S} \bm{f}(\bm{p}),
\end{equation*}

\n where in the last equation we used the fact that $\bm{p}$ is a finite sequence and that all terms after index $N$ are zero. Combining this with \eqref{A_N bound}, we obtain the bound

\begin{equation*}
\|\bm{A} \bm{F}(\bm{p})\|_X \leq \|\bm{A}^N \bm{F}(\bm{p})\|_X + \frac{1}{2(N+1)} \|\bm{\pi}_N \bm{S} \bm{f}(\bm{p})\|_X.
\end{equation*}

\n It remains to use the equation $\bm{\pi}_N \bm{S} \bm{f}(\bm{p}) = \left(\bm{\pi}^{2N+1} - \bm{\pi}^N \right) \bm{S} \bm{f}(\bm{p})$, which holds because the nonlinearities in $\bm{f}$ are quadratic and $\bm{p}$ is a finite sequence whose nonzero elements are at indices $n\leq N$. Indeed, all nonzero elements of $\bm{f}(\bm{p})$ occur at indices $n \leq 2N$, while those of $\bm{S} \bm{f}(\bm{p})$ occur at indices $n \leq 2N + 1$, due to the index shift in the operator $S$ \eqref{M S}.

\vspace{.1in}

$(ii)$ In general, $\bm{F}'(\bm{p}) : X \to Y$. However, the difference $\bm{F}'(\bm{p}+\bm{h}) - \bm{F}'(\bm{p}+\bm{\tilde h})$ is a bounded operator from $X$ into itself (recall that $X$ is a smaller space than $Y$). This is because, when taking this difference, the operator $\bm{M}$ whose range lies in $Y$, cancels. Indeed, in view of \eqref{F M S vector}

\begin{equation*}
\bm{F}'(\bm{p}+\bm{h}) - \bm{F}'(\bm{p}+\bm{\tilde h}) = \bm{S} \left[ \bm{f}'(\bm{p}+\bm{h}) - \bm{f}'(\bm{p}+\bm{\tilde h}) \right]
\end{equation*}

\n holds for indices $n \geq 1$. For $n=0$, $\bm{F}$ is given by the boundary conditions $\bm{B}$, which define a linear operator. Hence, the above difference of derivatives is zero. Moreover, by definition, $\bm{S}$ is equal to zero at $n=0$, so the above equality also holds for $n=0$. It is now clear that

\begin{equation*}
\|\bm{A} \left[ \bm{F}'(\bm{p}+\bm{h}) - \bm{F}'(\bm{p}+\bm{\tilde h}) \right]\|_{B(X)} \leq \|\bm{A} \|_{B(X)} \|\bm{S} \|_{B(X)} \|\bm{f}'(\bm{p}+\bm{h}) - \bm{f}'(\bm{p}+\bm{\tilde h}) \|_{B(X)}.
\end{equation*}

\n The result then follows by combining this with Lemma \ref{LEM S bound}, Lemma~\ref{LEM f} and  the estimate \eqref{A_N bound}.

\end{proof}

\subsubsection{The $Z_1$ bound}

\begin{lemma} \label{LEM Z_1}
Assume (H1) and (H2), then $\|I - \bm{A} \bm{F}'(\bm{p}) \|_{B(X)} \leq Z_1 = Z_{1,1} + Z_{1,2} + Z_{1,3}$, where 

\begin{equation*}
Z_{1,1} =  \| \bm{\pi}^N \left( I - \bm{A} \bm{F}'(\bm{p}) \right) \bm{\pi}^{2N+1}\|_{B(X)}, \qquad
Z_{1,2} = \frac{2}{\nu^{2N+2}} \|\bm{A}^N\|_{B(X)},
\qquad
Z_{1,3} = \frac{\nu C}{N+1},
\end{equation*}
with $C$ given by \eqref{C}.

\end{lemma}

\begin{proof}

Using the definition of the operator $\bm{A}$ and the formula \eqref{F'p}, we may write

\begin{equation*}
\bm{A}  \bm{F}'(\bm{p}) = \bm{A}^N  \bm{F}'(\bm{p}) + \bm{A}_N  \bm{F}'(\bm{p}) = \bm{A}^N  \bm{F}'(\bm{p}) + \bm{A}_N \bm{M} + \bm{A}_N \bm{S} \bm{f}'(\bm{p}).
\end{equation*}

\n Next $I = \bm{\pi}^N I + \bm{\pi}_N I$, and taking the difference of these two identities, and using that $\bm{A}_N \bm{M} = \bm{\pi}_N I$ (cf. \eqref{BM=Id}), we obtain

\begin{equation*}
\begin{split}
I - \bm{A}  \bm{F}'(\bm{p}) &= \bm{\pi}^N \left( I - \bm{A} \bm{F}'(\bm{p}) \right) - \bm{A}_N \bm{S} \bm{f}'(\bm{p}) =  
\\[.1in]
&= \bm{\pi}^N \left( I - \bm{A} \bm{F}'(\bm{p}) \right) \bm{\pi}^{2N+1} - \bm{A}^N \bm{R} - \bm{A}_N \bm{S} \bm{f}'(\bm{p}),
\end{split}
\end{equation*}

\n In the last step we used that $\bm{\pi}^N I \bm{\pi}_{2N+1} = 0$ and have set 

\begin{equation*}
\bm{R} = \bm{\pi}^N \bm{F}'(\bm{p}) \bm{\pi}_{2N+1}.
\end{equation*}

\n $\bm{F}'(\bm{p}) : X \to Y$ is a bounded operator, and it is easy to see that the projection $\bm{\pi}^N : Y \to X$ is also bounded, since the infinite sums reduce to finite ones. Consequently, $\bm{R}$ is a bounded operator on $X$, and therefore, using Lemma~\ref{LEM S bound}, part $(iii)$ of Lemma~\ref{LEM f} and the estimate \eqref{A_N bound}, we immediately obtain

\begin{equation*}
\|I - \bm{A}  \bm{F}'(\bm{p})\|_{B(X)} \leq Z_{1,1} + \|\bm{A}^N\|_{B(X)} \|\bm{R}\|_{B(X)} + \frac{\nu C}{N+1}
\end{equation*}

\n To conclude the proof we need to show that $\displaystyle \|\bm{R}\|_{B(X)} \leq 2 /  \nu^{2N+2}$. To that end we will show that

\begin{equation} \label{1}
\bm{\pi}^N \left[ \bm{M} + \bm{S} \bm{f}'(\bm{p}) \right] \bm{\pi}_{2N+1} = 0,
\end{equation}

\n which in view of \eqref{F'p} implies that

\begin{equation*}
\bm{R}  =
\begin{cases}
\bm{B} \bm{\pi}_{2N+1}, \qquad \qquad &n=0
\\
0, & n\geq 1.
\end{cases}
\end{equation*}

\n Recall that $\bm{B} : X \to \RR^8$ and for a given element of $X$ it returns a single 8-tuple of numbers given by \eqref{B 1-4} and \eqref{B 5-8}. By definition for any $\bm{w} = (w_{1}, ..., w_8) \in X = \left( \ell^1_\nu \right)^8$

\begin{equation*}
B_1 \bm{\pi}_{2N+1}  \bm{w} = \sum_{n=2N+2}^\infty \beta_n w_{2, n},
\end{equation*}

\n where the sequence $\beta_n$ is defined by \eqref{alpha beta}, and $w_{2,n}$ denotes the $n$-th term of the sequence $w_2$. Since $|\beta_n| \leq 2$ and $\nu \geq 1$ we may estimate

\begin{equation*}
\left| B_1 \bm{\pi}_{2N+1}  \bm{w} \right| = \left| \sum_{n=2N+2}^\infty \frac{\beta_n}{\nu^n} w_{2, n} \nu^n \right| \leq \frac{2}{\nu^{2N+2}} \|w_2\|_{\ell^1_\nu} \leq \frac{2}{\nu^{2N+2}} \|\bm{w}\|_X.
\end{equation*}

\n Obviously, the same estimate also holds with $B_2, ..., B_8$ in place of $B_1$, which then implies the desired estimate. To conclude the proof, it thus remains to verify \eqref{1}. The definition \eqref{M S} immediately gives that $\bm{\pi}^N \bm{M} \bm{\pi}_{2N+1} = 0$. On the other hand, part $(ii)$ of Lemma~\ref{LEM f} implies that

\begin{equation*}
\bm{f}'(\bm{p}) \bm{\pi}_{2N+1} = \bm{\pi}_{N+1} \bm{f}'(\bm{p}) \bm{\pi}_{2N+1} 
\end{equation*}

\n Consequently, \eqref{1} follows from the identity

\begin{equation} \label{2}
\bm{\pi}^N \bm{S} \bm{\pi}_{N+1} = 0,   
\end{equation}

\n which in turn follows directly from the definition of the operator $\bm{S}$ (see \eqref{M S}).
\end{proof}

\subsubsection{Existence of a solution for $\omega=0$} \label{SECT omega=0}

We now conclude the application of Theorem~\ref{THM NK2} to the operator $\bm{F}$ \eqref{F_j,n} for $\omega=0$, using the constants $Y_0$, $Z_1$ and $Z_2$ obtained in the previous sections, and we verify the hypothesis that for some $r>0$,

\begin{equation*}
Q(r) = \frac{1}{2} Z_2 r^2 - (1-Z_1) r + Y_0 < 0.
\end{equation*}

\n Theorem~\ref{THM NK2} then guarantees the existence of a point $\bm{x} \in B_r(\bm{p})$ such that $\bm{F}(\bm{x}) = 0$. As discussed in Section~\ref{SECT p and A}, our starting point is to choose $\bm{p} \in X$ such that $\bm{F}(\bm{p}) \approx 0$. We do this by considering, equivalently, the ODE system \eqref{1st order system}. We use the shooting method to obtain an unrefined or rough approximate solution $\bm{\pp}(t)$, followed by an interpolation process to obtain the approximate Chebyshev coefficients $\bm{p}$. We then apply Newton's method to refine the approximation. Thus, we obtain a finite (8-tuple of sequences) $\bm{p}$ of size $N$, i.e., $\bm{\pi}_N \bm{p} = 0$, and in particular, the assumption (H1) of Section~\ref{SECT p and A} holds. The truncation size $N$ is determined by our numerical approximation, and we take $N=70$. The code and numerical tools needed to compute the necessary quantities are available on GitHub at \cite{TransmissionEigenvalues.jl}.

We next take $\bm{A}^N$ to be a numerical inverse of the finite part $\bm{\pi}^N \bm{F}'(\bm{p}) \bm{\pi}^N$, which in particular satisfies the assumption (H2) of Section~\ref{SECT p and A}. Choosing $\nu = 1.08$, we estimate

\begin{equation*}
\|\bm{A}^N\|_{B(X)} \leq 27.08.
\end{equation*}

\n Using Lemmas~\ref{LEM Y0 and Z_2} and \ref{LEM Z_1} we further estimate

\begin{equation*}
Y_0 \leq 2.5 \cdot 10^{-12}, \qquad Z_2 \leq 58.5 \qquad \text{and} \qquad Z_1 \leq 0.965.
\end{equation*}

\n Taking $r = 10^{-10}$ we check that $Q(r) < -0.9 \cdot 10^{-12} < 0$. Finally, we also check that 

\begin{equation*}
\|\bm{p}\|_X \geq 112.33 \geq r.
\end{equation*}

\n This guarantees that the origin lies outside of the ball $B_r(\bm{p})$, and the obtained solution $\bm{x}$ is nontrivial. For our implementation, we used the Julia programming language \cite{Julia-2017} and the RadiiPolynomial.jl \cite{RadiiPolynomial.jl} package, which allows for easy manipulation of sequences. We also use the interval arithmetic package IntervalArithmetic.jl \cite{IntervalArithmetic.jl} to rigorously evaluate mathematical expressions.

Thus, we conclude that the system \eqref{TE u reduced} for $\omega=0$ has a nontrivial solution satisfying the symmetry \eqref{sym}.

\subsection{Existence of solutions for $\omega>0$} \label{SECT omega>0}

Our goal now is to prove the existence of a branch of (nontrivial) solutions to the equation $\bm{F}_\omega(\bm{x})=0$, or equivalently to \eqref{TE u reduced} with \eqref{sym}, for $\omega \geq 0$. Note that, to emphasize the dependence of the operator on the parameter $\omega$, we use the notation $\bm{F}_\omega$. For methods to perform rigorous continuation with respect to a parameter, we refer to \cite{polynomial_chaos, cadiot_witham, marschal, breden:tel-05142432}. The approach is again based on a parameter-dependent version of the Newton-Kantorovich Theorem~\ref{THM NK2}, which we state below and whose proof is a precise analogue of that of Theorem~\ref{THM NK2}.

\begin{theorem} \label{THM NK3}
Let $X, Y$ be Banach spaces, $r_0 > 0$ and $b>a$ be given. For each $\omega \in [a, b]$ let $p_\omega \in X$, $F_\omega : X \to Y$ be Fr{\'e}chet differentiable in the ball $B_{r_0}(p_\omega)$ and let $A_\omega : Y \to X$ be an injective, bounded linear operator. Assume that $Y_0, Z_1$ and $Z_2$ are nonnegative constants such that 

\begin{enumerate}
\item[(i)] $\displaystyle \sup_{\omega \in [a, b]} \|A_\omega F_\omega (p_\omega)\|_X \leq Y_0$,
\item[(ii)] $\displaystyle \sup_{\omega \in [a, b]} \|I - A_\omega F_\omega'(p_\omega) \|_{B(X)} \leq Z_1$,
\item[(iii)] $\displaystyle \sup_{\omega \in [a, b]}\|A_\omega [ F_\omega'(p_\omega+h) - F_\omega'(p_\omega+\tilde h) ]\|_{B(X)} \leq Z_2 \|h - \tilde{h}\|_X$ for all $h, \tilde h \in B_{r_0}(0)$.
\end{enumerate}

\n Consider the quadratic

\begin{equation*}
Q(r) = \frac{1}{2} Z_2 r^2 - (1-Z_1) r + Y_0.
\end{equation*}

\n If there exists $r \in (0, r_0)$ such that $Q(r)<0$, then for every $\omega \in [a, b]$ there exists a unique $x_\omega \in B_r(p_\omega)$ such that $F_\omega(x_\omega) = 0$. 

\end{theorem}

We are going to apply the above theorem to our operator $\bm{F}_\omega$ between the spaces $X$ and $Y$ as defined in Sections~\ref{SECT Cheb exp} and \ref{SECT function spaces}. To avoid multiple subscripts, instead of $\bm{p}_\omega$ and $\bm{A}_\omega$, we will write $\bm{p}(\omega)$ and $\bm{A}(\omega)$, respectively. For the construction of these quantities, we will use Chebyshev expansions with respect to the variable $\omega \in [a, b]$. To that end, let us change the variable to $t \in [-1,1]$ using

\begin{equation*}
\omega = \frac{b - a}{2} t + \frac{b + a}{2} \qquad \Longleftrightarrow \qquad t = 2\frac{\omega - b}{b - a} + 1.
\end{equation*}

\n To construct $\bm{p}(\omega)$, we first fix an integer $N_{\text{cheb}} \geq 1$ and discretize the interval $[a, b]$ using the Chebyshev nodes

\begin{equation*}
\omega_n = \frac{b - a}{2} \cos \frac{(2n-1)\pi}{2N_{\text{cheb}}} + \frac{b + a}{2}, \qquad \qquad n=1,...,N_{\text{cheb}}.
\end{equation*}

\n For each of these finitely many values, we can find $\bm{p}(\omega_n) = \left(p_1(\omega_n), ..., p_8(\omega_n) \right) \in X$ and a finite-part operator $\bm{A}^N(\omega_n)$, precisely as in Section~\ref{SECT omega=0}, except that instead of $\omega=0$, we now perform the calculations with $\omega = \omega_n$. In particular, each $p_j(\omega_n) \in \ell^1_\nu$ is a finite sequence of size $N$. We now define $p_j(\omega)$, for $j=1,...,8$, to be a finite Chebyshev series 

\begin{equation*}
p_j(\omega) = c_{j,0} + 2 \sum_{n=1}^{N_{\text{cheb}}} c_{j,n} T_n \left( 2\frac{\omega - b}{b - a} + 1 \right),
\end{equation*}

\n where each of the coefficients $c_{j,n} \in \ell^1_\nu$ is also a finite sequence. These coefficients are obtained by fitting the above function to the calculated set of values $\left\{ p_j(\omega_n) : n=1,...,N_{\text{cheb}} \right\}$ using FFT. We then set $\bm{p}(\omega) = \left(p_1(\omega), ..., p_8(\omega) \right)$, which clearly lies in $X$ for any $\omega \in [a,b]$. Next, since $|T_n(t)| \leq 1$ for $t \in [-1,1]$, we have the estimate

\begin{equation*}
\sup_{\omega \in [a, b]} \|\bm{p}(\omega)\|_X \leq \max_{j=1,...,8} \left\{ \|c_{j,0}\|_{\ell^1_\nu} + 2 \sum_{n=1}^{N_{\text{cheb}}} \|c_{j,n}\|_{\ell^1_\nu} \right\}.
\end{equation*}

\n Similarly, 

\begin{equation*}
\bm{A}^N(\omega) = \bm{a}_{j,0} + 2 \sum_{n=1}^{N_{\text{cheb}}} \bm{a}_{j,n} T_n \left( 2\frac{\omega - b}{b - a} + 1 \right),
\end{equation*}

\n where each $\bm{a}_{j,n} : X \to X$ is a finite-part operator (and can be represented as a matrix) and is computed in a similar way from the values $\bm{A}^N(\omega_n)$ for $n=1,...,N_{\text{cheb}}$. We again have the uniform estimate

\begin{equation*}
\sup_{\omega \in [a, b]} \|\bm{A}^N(\omega)\|_{B(X)} \leq \max_{j=1,...,8} \left\{ \|\bm{a}_{j,0}\|_{B(X)} + 2 \sum_{n=1}^{N_{\text{cheb}}} \|\bm{a}_{j,n}\|_{B(X)} \right\}.
\end{equation*}

Finally, we define $\bm{A}(\omega) = \bm{A}^N(\omega) + \bm{A}_N$, where the tail operator $\bm{A}_N$ is defined as in (H2) in Section~\ref{SECT p and A} and is independent of $\omega$. With the above two estimates and our $Y_0$, $Z_1$ and $Z_2$ bounds from Lemmas~\ref{LEM Y0 and Z_2} and \ref{LEM Z_1}, we now apply Theorem~\ref{THM NK3}. More specifically, we take $a=0$, $b=2.03$, choose $N_{\text{cheb}} = 50$, $N = 90$ and $\nu = 1.08$. We then obtain

\begin{equation*}
\sup_{\omega \in [a, b]} \|\bm{A}^N(\omega)\|_{B(X)} \leq 27.6
\end{equation*}

\n and

\begin{equation*}
Y_0 \leq 2 \cdot 10^{-9}, \qquad Z_2 \leq 59.6 \qquad \text{and} \qquad Z_1 \leq 0.7.
\end{equation*}

\n Taking $r=0.005$ gives $Q(r) \leq -0.0007$. Consequently, Theorem~\ref{THM NK3} implies that for each $\omega \in [0,2.03]$ there exists $\bm{x}(\omega) \in X$ such that $\bm{F}_\omega(\bm{x}(\omega)) = 0$. Moreover,

\begin{equation*}
\sup_{\omega \in [0, 2.03]} \|\bm{x}(\omega) - \bm{p}(\omega)\|_X \leq 0.005.
\end{equation*}

\n We also estimate 

\begin{equation*}
\inf_{\omega \in [0, 2.03]} \|\bm{p}(\omega)\|_X \geq 129.4 \geq r,
\end{equation*}

\n which implies that $\bm{x}(\omega)$ is a nontrivial solution. Finally, the mapping $\omega \mapsto \bm{x}(\omega)$ is smooth because all the quantities $\bm{p}(\omega)$, $\bm{A}(\omega)$ and $\bm{F}_\omega(\bm{p}(\omega))$ depend smoothly on $\omega$, as they are given by finite Chebyshev sums with respect to $\omega$. Further, the operator $\bm{F}_\omega(\bm{x})$ depends smoothly on $(\omega, \bm{x})$. Consequently, its zero $\bm{x}(\omega)$, obtained through a fixed-point equation, also depends smoothly on $\omega$. From the above estimates on $\bm{p}(\omega)$ and $\bm{x}(\omega)$, and the fact that $\bm{\pi}^N \bm{p}(\omega)= \bm{p}(\omega)$ it follows that $\bm{\pi}^N \bm{x}(\omega)$ is uniformly bounded away from zero in any norm. We may conclude that the functions $\bm{\xx}(\omega)$ are uniformly bounded away from $0$ in $L^2(D)\times L^2(D)$ and therefore in $H^2(D)\times L^2(D)$ for $\omega \in [0,2.03]$.

\section*{Acknowledgments}
{The research of FC was partially supported by the NSF Grant DMS--24--06313. The work of MSV was partially supported by NSF grant DMS--22--05912.}

\bibliographystyle{abbrv}
\bibliography{refs}
\end{document}